\documentclass[12pt]{amsart}
\usepackage{amssymb,amsbsy,amsthm,amscd}
\usepackage[mathcal]{eucal}

\usepackage[all]{xy}

\theoremstyle{plain}
\newtheorem{thm}{Theorem}[section]

\newtheorem{lem}[thm]{Lemma}

\newtheorem{prop}[thm]{Proposition}

\newtheorem{prb}[thm]{Problem}
\newtheorem{con}[thm]{Conjecture}

\numberwithin{equation}{section}

\newcommand{\N}{\ensuremath{\mathbb N}}

\newcommand{\Z}{\ensuremath{\mathbb Z}}

\DeclareMathOperator{\h}{H}
\DeclareMathOperator{\ab}{ab}

\DeclareMathOperator{\GL}{GL}
\DeclareMathOperator{\PG}{P\Gamma}
\DeclareMathOperator{\SL}{SL}
\DeclareMathOperator{\PSL}{PSL}

\DeclareMathOperator{\Aut}{Aut}

\DeclareMathOperator{\IA}{IA}

\begin{document}
\date\today
\title{On the Abelianization of\\Congruence Subgroups of $\Aut(F_2)$}
\thanks{The author would like to thank F.~Grunewald for proposing the topic and B.~Klopsch for many helpful discussions.}
\author{Daniel Appel}
\address{Department of Mathematics, Royal Holloway College, University of London, Egham, Surrey, TW20 0EX,
United Kingdom.}
\curraddr{Mathematisches Institut der Heinrich- Heine- Universit\"at\\ 40225 D\"usseldorf\\ Germany.}
\email{Daniel.Appel@uni-duesseldorf.de}

\maketitle
\pagestyle{plain}
\begin{abstract}
Let $F_n$ be the free group of rank $n$ and let $\Aut^+(F_n)$ be its special automorphism group. For an epimorphism $\pi : F_n \rightarrow G$ of the free group $F_n$ onto a finite group $G$ we call $\Gamma^+(G,\pi)=\{ \varphi \in \Aut^+(F_n) \mid \pi\varphi~=~\pi \}$ the standard congruence subgroup of $\Aut^+(F_n)$ associated to $G$ and $\pi$.
In the case $n = 2$ we fully describe the abelianization of $\Gamma^+(G,\pi)$ for finite abelian groups $G$. Moreover, we show that if $G$ is a finite non-perfect group, then $\Gamma^+(G,\pi) \leq \Aut^+(F_2)$ has infinite abelianization.
\end{abstract}

\section{Introduction}
\subsection{Main Results}
Let $F_n$ be the free group on $n$ generators and $\Aut(F_n)$ its
group of automorphisms. Moreover, let $\pi: F_n \rightarrow G$ be an
epimorphism of $F_n$ onto a finite group $G$. The automorphism group $\Aut(F_n)$ acts in a natural way on the (finite) set of all epimorphisms from $F_n$ onto $G$. By 
$$
\Gamma(G,\pi) := \{ \varphi \in \Aut(F_n) \mid  \pi \varphi =\pi\}
$$
we denote the stabilizer of $\pi$ under this action. The group $\Gamma(G,\pi)$ is called the \emph{standard congruence subgroup of $\Aut(F_n)$ associated to
$G$ and~$\pi$}. This is a finite index subgroup of $\Aut(F_n)$. A subgroup of $\Aut(F_n)$ containing some $\Gamma(G,\pi)$ is called a \emph{congruence subgroup of~$\Aut(F_n)$}. 

A classical question is whether every finite index subgroup of $\Aut(F_n)$ is a congruence subgroup. Quite recently it has been shown that every finite index subgroup of $\Aut(F_2)$ is a congruence subgroup. See~\cite{MA} and~\cite{BER}. For $n \geq 3$ this question is still open.

The groups $\Gamma(G,\pi)$ have been studied by various authors. For instance, in~\cite{GL}, F.~Grunewald and A.~Lubotzky use the groups $\Gamma(G,\pi)$ to construct linear representations of the automorphism group $\Aut(F_n)$. Our work is related to results of T.~Satoh~\cite{SA1, SA2}, see also Section~\ref{comments}. The joint work~\cite{AR} of E.~Ribnere and the author can be seen as an accompanying paper to the present one.

The automorphism group $\Aut(F_n)$ has  a well-known representation onto $\GL_n(\Z)$ given by
$$
\rho : \Aut(F_n) \longrightarrow \Aut(F_n / F'_n) \cong \GL_n(\Z),
$$
where $F'_n$ denotes the commutator subgroup of $F_n$, see~\cite[Sec.~3.6]{MKS} for details.  Its kernel is denoted by $\IA_n$ and called the
\emph{group of $\IA_n$-automorphisms} or sometimes also the \emph{classical Torelli group}. For an interesting generalization see~\cite{MS}.

As one classically considers $\SL_n(\Z)$ instead of $\GL_n(\Z)$, we focus on the \emph{special automorphism group} $\Aut^+(F_n) := \rho^{-1}(\SL_n(\Z))$, which is a subgroup of index~$2$ in $\Aut(F_n)$. We set $\Gamma^+(G,\pi) := \Gamma(G,\pi) \cap \Aut^+(F_n)$. This is a subgroup of index at most~$2$ in $\Gamma(G,\pi)$. The term \emph{congruence subgroup of} $\Aut^+(F_n)$ is defined in the obvious way.

In this paper we study the abelianizations $\Gamma^+(G,\pi)^{\ab}$ of the groups $\Gamma^+(G,\pi)$ in the case $n = 2$. We remark that if $G$ is abelian, then, up to conjugation, $\Gamma^+(G,\pi)$ depends only on $G$ but not on the particular epimorphism $\pi:F_2 \rightarrow G$, see \cite[Lem.~3.1]{AR}. Moreover, every finite abelian group generated by two elements can be written as $\Z/m\Z \times \Z/n\Z$ where $n,m \in \N$ such that $n \mid m$. The following theorem therefore covers all possible choices for an epimorphism $\pi: F_2 \rightarrow G$ onto a finite abelian group. 

\begin{thm}\label{main1}
Let $m, n \in \N$ such that $m \geq 3$, $n \mid m$ and $(m,n) \not= (3,1)$. Let $G:=\Z/m\Z \times \Z/n\Z$ and $\pi: F_2 \rightarrow G$ be an epimorphism. Then
$$ \Gamma^+(G, \pi)^{\ab} \cong G \times \Z^{ 1 + 12^{-1}nm^2\prod_{p\mid m}(1-p^{-2}) }$$
where the product runs over all primes $p$ dividing $m$.

Furthermore, we have
\begin{align*}
\Gamma^+(\Z / 2\Z,\pi)^{\ab} &\cong \Z / 2\Z \times \Z / 4\Z \times \Z,\\
\Gamma^+(\Z / 3\Z,\pi)^{\ab} &\cong \Z / 3\Z \times \Z / 3\Z \times \Z,\\
\Gamma^+(\Z / 2\Z \times \Z / 2\Z,\pi)^{\ab} &\cong \Z / 2\Z \times \Z / 2\Z \times \Z / 2\Z \times \Z^2.
\end{align*}
\end{thm}

\begin{thm}\label{main2}
Let $\pi : F_2 \rightarrow G$ be an epimorphism of $F_2$ onto a finite non-perfect group $G$. Then $\Gamma^+(G,\pi)$ has infinite abelianization.
\end{thm}

For $n,m \in \N$ with $n \mid m$ we define a subgroup of $\SL_2(\Z)$ by
$$\Gamma(m,n) := \{ \left(
\begin{smallmatrix} a & b \\ c & d\end{smallmatrix} \right) \in
\SL_2(\Z)\mid a \equiv_m 1 , b \equiv_m 0  \text{ and } c \equiv_n
0 , d \equiv_n 1 \}.$$
By $\PG(m,n)$ we denote the image of $\Gamma(m,n)$ in $\PSL_2(\Z)$ under the natural projection. One of the main ingredients in our proofs is 
\begin{prop}\label{gamma10free}
Let $m,n \in \N$ such that $m \geq 3$, $n\mid m$ and $(m,n) \not= (3,1)$. Then $\Gamma(m,n)$ and $\PG(m,n)$ are free of rank 
$$1 + \frac{nm^2}{12}\prod_{\substack{p \mid m\\ p \ \mathrm{prime}}} \left(1- \frac{1}{p^2}\right).$$
\end{prop} 

In particular, for primes $p \geq 5$, the groups $\Gamma(p,1) \leq \SL_2(\Z)$ are free of rank $1+ \frac{1}{12}p^2(1-p^{-2})$ so that the rank of $\Gamma(p,1)$ grows quadratically in $p$. In contrast, for $n \geq 3$, one can show that the corresponding subgroups in $\SL_n(\Z)$ can always be generated by $n(n-1)$ matrices. See \cite[Lem.~6.1]{GL}.

\subsection{Related Results and Open Problems}\label{comments}
In \cite{GL} Grunewald and Lubotzky use the groups $\Gamma(G,\pi)$ to construct linear representations of the automorphism group $\Aut(F_n)$. In their concluding Section~9.4 they present, for some explicit $G$ of small order, the indices of the groups $\Gamma^+(G,\pi)$ in $\Aut^+(F_n)$ and also the abelianizations of the groups $\Gamma^+(G,\pi)$ which they obtain by MAGMA computations. Besides the case that $G$ is finite abelian, they also consider the case that $G = D_r$ is a dihedral group. Their observations are explained by another result of the author, which says 
$$\Gamma^+(D_r,\pi)^{\ab} \cong \begin{cases} 
	\Z/2\Z \times \Z^2, \quad r \mbox{ odd} \\ 
	\Z/2\Z \times \Z^3, \quad r \mbox{ even} \end{cases}.$$
The proof of this result is elementary but rather long. It shall therefore be postponed to the author's Ph.D.~Thesis.

Actually, Grunewald and Lubotzky present computational results for some more finite groups $G$, e.g., $G = A_5$. In all considered cases $\Gamma^+(G,\pi)$ has infinite abelianization. For $G$ non-perfect we now know by Theorem $\ref{main2}$ that $\Gamma^+(G,\pi)^{\ab}$ is infinite. However, our proof does not work for perfect groups $G$. Hence we state
\begin{con}
For every epimorphism $\pi:F_2 \rightarrow G$ onto a non-trivial finite group $G$ the group $\Gamma^+(G,\pi) \leq \Aut^+(F_2)$ has infinite abelianization.
\end{con}
The situation in the case $n \geq 3$ looks different. Indeed, Grunewald and Lubotzky show that for every epimorphism $\pi : F_n \rightarrow G$ from $F_n$, $n\geq 3$, onto a finite abelian group, the group $\Gamma(G,\pi) \leq \Aut(F_n)$ has finite abelianization. This is Proposition~8.5 in~\cite{GL}. Computational results \cite[Sec.~9.4]{GL} indicate that $\Gamma^+(G,\pi)$ always has finite abelianization if $n \geq 3$. This leads to

\begin{prb}
Does $\Gamma^+(G,\pi) \leq \Aut^+(F_n)$, where $n \geq 3$, have finite abelianization for every epimorphism $\pi : F_n \rightarrow G$ onto a non-trivial finite group $G$?
\end{prb}

Our work is related to results of Satoh \cite{SA1, SA2}. In his papers Satoh considers the kernel $T_{n,m}$ of the composition
$$\Aut(F_n) \stackrel{\rho}{\longrightarrow} \GL_n(\Z) \longrightarrow \GL_n(\Z/ m\Z).$$
One easily sees that for $m \geq 3$ we have $T_{n,m} = \Gamma^+( (\Z / m\Z)^n,\pi)$ where $\pi: F_n \rightarrow (\Z/m\Z)^n$ is the obvious epimorphism. Satoh shows that for $n,m \geq 2$ one has
$$T^{\ab}_{n,m} \cong (\IA_n^{\ab} \otimes_{\Z} \Z / m \Z) \times \Gamma_n(m)^{\ab}$$
where $\Gamma_n(m)$ is the kernel of the natural epimorphism $\GL_n(\Z) \rightarrow \GL_n(\Z/m\Z)$. Since $\IA_2$ is free of rank $2$, see \cite[Sec.~3.6, Cor.~N4]{MKS}, for $n = 2$ this reads
$$T^{\ab}_{2,m} \cong ( \Z / m \Z)^2 \times \Gamma_2(m)^{\ab}.$$
Observe that for $m\geq 3$ we have $\Gamma_2(m) = \Gamma(m,m) \leq \SL_2(\Z)$. This result therefore corresponds to our result in Theorem \ref{main1} for the special case $G = (\Z / m\Z)^2$. Satoh also gives the integral homology groups of $T_{2,p}$ for odd primes $p$. In particular, he shows that 
$$\h_1(T_{2,p},\Z) = (\Z / p\Z)^2 \times \Z^{1+12^{-1}p^3(1-p^{-2})}.$$
Since the first integral homology group is just the abelianization, this corresponds to our result in Theorem \ref{main1} for the special case $G = (\Z/p\Z)^2$.

\section{Congruence Subgroups of $\SL_2(\Z)$}
\subsection{Introduction}
Let $\pi : F_2 \rightarrow G$ be an epimorphism of the free group $F_2$ onto a finite group $G$. In order to understand the image $\rho(\Gamma^+(G,\pi))$ in $\SL_2(\Z)$, we introduce some families of finite index subgroups of $\SL_2(\Z)$.

Recall that for two natural numbers $m,n \in \N$ such that $n \mid m$ we define the group
$$\Gamma(m,n) = \{ \left(
\begin{smallmatrix} a & b \\ c & d\end{smallmatrix} \right) \in
\SL_2(\Z)\mid a \equiv_m 1 , b \equiv_m 0  \text{ and } c \equiv_n
0 , d \equiv_n 1 \}.$$ 
For $m \in \N$ the group $\Gamma(m) := \Gamma(m,m)$ is called the \emph{principal congruence subgroup of $\SL_2(\Z)$ of level $m$}. It is the kernel of the natural epimorphism $\SL_2(\Z) \rightarrow \SL_2(\Z/m\Z)$. A subgroup of $\SL_2(\Z)$ containing some $\Gamma(m)$ is called a \emph{congruence subgroup of} $\SL_2(\Z)$. It is a classical result that not all finite index subgroups of $\SL_2(\Z)$ are congruence subgroups. See for example \cite{GJ}. 

In \cite{AR} it is shown that 
\begin{equation}\label{index}
[\SL_2(\Z) : \Gamma(m,n)] = n m^2 \prod_{p \mid m} \left(1-\frac{1}{p^2}\right)
\end{equation}
where the product runs over all primes $p$ dividing $m$.

It is convenient to use the following notation.
\begin{align*}
\Gamma^0(m)  &:= \{ A \in \SL_2(\Z) \mid A \equiv \left( \begin{smallmatrix} * & 0 \\ * & *\end{smallmatrix} \right) \mod m \},\\
\Gamma^{1}(m)&:= \{ A \in \SL_2(\Z) \mid A \equiv \left( \begin{smallmatrix} 1 & 0 \\ * & *\end{smallmatrix} \right) \mod m \}.
\end{align*}
All the above subgroups of $\SL_n(\Z)$ are clearly congruence subgroups. 

We denote the images in $\PSL_2(\Z)$ of $\Gamma^0(m)$, $\Gamma^1(m)$ and $\Gamma(m)$ under the natural projection $\SL_2(\Z) \rightarrow \PSL_2(\Z)$ by $\PG^0(m)$, $\PG^{1}(m)$ and $\PG(m)$, respectively.

\subsection{Free Congruence Subgroups}
In \cite{FR} H.~Frasch gives the following description of the groups $\PG(p)$ for $p$ prime. 

\begin{thm}[Frasch]\label{frasch}
Let $p \geq 3$ be a prime. Then $\PG(p)$ is free of rank $1 + \frac{1}{12} p^3 (1 - p^{-2})$. Moreover $\PG(2)$ is free of rank $2$.
\end{thm}

We shall now generalize his result. Consider the natural projection $\SL_2(\Z) \rightarrow \PSL_2(\Z)$. By definition it maps $\Gamma(m)$ onto $\PG(m)$. For $m \geq 3$ the kernel $\langle \left(\begin{smallmatrix} -1 & 0\\ 0 & -1 \end{smallmatrix} \right) \rangle$ of this projection has trivial intersection with $\Gamma(m)$. Hence we obtain an isomorphism
$$\Gamma(m) \stackrel{\cong}{\longrightarrow} \PG(m).$$
Note that this argument does not work for the case $m = 2$. Indeed, in contrast to $\PG(2)$, the group $\Gamma(2)$ is not free but the direct product of a rank $2$ free group and a cyclic group of order $2$. This is the reason for the third exceptional case in Theorem~\ref{main1}

\begin{lem}\label{free1}
Let $m \geq 3$. Then $\Gamma(m) \cong \PG(m)$ is free of rank
$$1 + \frac{m^3}{12} \prod_{p \mid m} \left( 1 - \frac{1}{p^2} \right)$$
where the product runs over all primes $p$ dividing $m$.
\end{lem}

\begin{proof}
Observe that for $m_1, m_2 \in \N$ such that $m_1 \mid m_2$ we have $\Gamma(m_2) \leq \Gamma(m_1)$. Accordingly, the main point in our proof is that a subgroup of a free group is again free and its rank is given by the Schreier Formula \cite[Thm.~2.10]{MKS}. Since $\Gamma(2)$ is not free, we consider two cases.\smallskip\\
Case 1: We have $m = 2^a$ for some $a \geq 2$. One can verify that $\PG(4)$ has index $4$ in $\PG(2)$. Since the latter group is free of rank $2$, we find that $\PG(4)$ is free of rank $5$ and hence so is $\Gamma(4)$. From (\ref{index}) we know that $\Gamma(2^a)$ has index $2^{3a-6}$ in $\Gamma(4)$. We thus find that $\Gamma(2^a)$ is free of rank $1 + \frac{1}{12}2^{3a}(1 - 2^{-2})$, as claimed.\smallskip\\
Case 2: We have $p_0 \mid m$ for some prime $p_0 > 2$. Again by (\ref{index}), we see that
$$[\Gamma(p_0) : \Gamma(m)] = \frac{m^3}{p_0^3} \prod_{\substack{p \mid m\\ p \not= p_0}} \left(1 - \frac{1}{p^2}\right).$$
Since, by Proposition~\ref{frasch}, $\Gamma(p_0)$ is free of rank $1 + \frac{1}{12}p_0^3(1 - p_0^{-2})$, we obtain the desired result.
\end{proof}

We next wish to generalize this result even further to the groups $\Gamma(m,n)$ where $m \geq 3$, $n \mid m$ and $(m,n) \not= (3,1)$. The first step is, of course, to show that these groups are free. Here we use the well-known description of $\PSL_2(\Z)$ as a free product
$$\PSL_2(\Z) = \left\langle \begin{pmatrix} 0 & 1 \\ -1 & 0 \end{pmatrix} \right\rangle * \left\langle \begin{pmatrix} 0 & -1 \\ 1 & 1 \end{pmatrix}  \right\rangle$$
where the first factor has order~$2$ and the second one has order~$3$. From the Kurosh Subgroup Theorem \cite[Cor.~4.9.1]{MKS} it follows that every non-trivial element of finite order in $\PSL_2(\Z)$ has either order~$2$ or~$3$.

Using this observation and considering some minimal cases, we obtain
\begin{lem}\label{free2}
Let $m \geq 4$, then $\PG^{1}(m)$ is a free group.
\end{lem}

\begin{proof}
First we consider the case that $m$ has a prime factor $p \geq 5$. Note that it suffices to show that $\PG^1(p)$ is free. To this end, we show that $\PG^1(p)$ does not contain a non-trivial element of finite order. Then the Kurosh Subgroup Theorem yields the desired result. Assume that $A \in \PG^1(p)$ is a non-trivial element of finite order. By definition of $\PG^1(p)$ we have
$$ A \equiv \begin{pmatrix} 1 & 0 \\ k & 1\end{pmatrix} \mod p$$
for some $0 \leq k \leq p-1$. It follows that
$$ A^p \equiv \begin{pmatrix} 1 & 0 \\ 0 & 1\end{pmatrix} \mod p.$$
Now we consider two cases.\smallskip\\
Case 1: $A^p = 1$. Then the order of $A$ divides $p$ and we have a contradiction, since $A$ has either order $2$ or $3$.\smallskip\\
Case 2: $A^p \not= 1$. Then $A^p$ is a non-trivial element of $\PG(p)$, which is, by Theorem \ref{frasch}, a free group. Hence $A^p$ does not have finite order, contradiction.

To prove the result for $m$ not having a prime factor $\geq 5$, it suffices to consider the cases where $m= 4,6,9$. By an explicit computation, using the Reidemeister-Schreier Method, one verifies that $\PG^1(m)$ is also free in these cases.
\end{proof}

From Lemmas \ref{free1}, \ref{free2}, the formulas (\ref{index}) and the Schreier Formula, we obtain Proposition \ref{gamma10free}.

H.~Rademacher \cite{RA} gives a very explicit description of the groups $\PG^0(p)$ for $p$ prime. Observe that for $p \in \{ 2,3 \}$ the groups $\PG^0(p)$ and $\PG^1(p)$ coincide. By two examples of Rademacher we have
\begin{align*}
\PG^1(2) &\cong \Z * \Z / 2 \Z\\
\PG^1(3) &\cong \Z * \Z / 3 \Z.
\end{align*}
In particular, $\PG^1(2)$ and $\PG^1(3)$ are not free. This is the reason for the first two exceptional cases in Theorem \ref{main1}.

\section{Proofs of the Main Results}
\subsection{Proof of Theorem \ref{main1}}
Let $G = \Z / m\Z \times \Z / n \Z$ where $m \geq 3$, $n \mid m$ and $(m,n) \not= (3,1)$. Then, by \cite[Lem.~3.1]{AR}, up to conjugation, the group $\Gamma^+(G,\pi)$ only depends on $G$ but not on the particular choice of the epimorphism $\pi : F_2 \rightarrow G$. We may thus suppose that $\pi(x) = (1,0)$ and $\pi(y) = (0,1)$. We have an exact sequence
$$1 \longrightarrow \IA_2 \longrightarrow \Gamma^+(G,\pi) \stackrel{\rho}{\longrightarrow} \Gamma(m,n) \longrightarrow 1.$$
By Proposition \ref{gamma10free}, the group $\Gamma(m,n)$ is free of rank
$$r := 1 + \frac{nm^2}{12} \prod_{p \mid m} \left( 1 - \frac{1}{p^2} \right).$$
Let $\{M_i = \left(\begin{smallmatrix} a_i & b_i \\ c_i & d_i  \end{smallmatrix}\right)\mid  1 \leq i \leq r\}$ be a set of free generators of $\Gamma(m,n)$ and write $F_2 = \langle x,y \rangle$. Moreover, let $\varphi_i \in \Aut^+(F_2)$ such that $\rho(\varphi_i) = M_i$, that is
\begin{equation}\label{koeff}
\varphi_i(x) \equiv a_i x + c_i y,\quad \varphi_i(y) \equiv b_i x + d_i y \mod F'_2
\end{equation}
where $F'_2$ denotes the commutator subgroup of $F_2$. By a classical result of J.~Nielsen \cite[Sec.~3.6,~Cor.~N4]{MKS} the group $\IA_2$ is free on $\alpha_x, \alpha_y$, the inner automorphisms given by conjugation with $x$ and $y$, respectively. Now a result of P.~Hall, see \cite[Ch.~13, Thm.~1]{DJ}, yields that $\Gamma^+(G,\pi)$ admits a presentation
$$\langle \alpha_x, \alpha_y, \varphi_1, \dots, \varphi_r \mid \varphi_i \alpha_x \varphi_i^{-1} = w_i,\ \varphi_i \alpha_y \varphi_i^{-1} = v_i \mbox{ for } 1 \leq i \leq r\rangle$$
where the $w_i$ and $v_i$ are suitable words in $\alpha_x, \alpha_y$. We have 
$$\varphi_i \alpha_x \varphi_i^{-1} = \alpha_{\varphi_i(x)},\quad \varphi_i \alpha_y \varphi_i^{-1} = \alpha_{\varphi_i(y)}$$
for $1 \leq i \leq r$. Hence, from (\ref{koeff}) it follows that
\begin{equation}\label{rel01}
\overline{\alpha_x} = a_i \overline{\alpha_x} + c_i \overline{\alpha_y},\quad \overline{\alpha_y} = b_i \overline{\alpha_x} + d_i \overline{\alpha_y}
\end{equation}
in the abelianization of $\Gamma^+(G,\pi)$. This yields that $\Gamma^+(G,\pi)^{\ab}$ is the abelian group generated by $\overline{\alpha_x}, \overline{\alpha_y}$ and $\overline{\varphi_i}, 1 \leq i \leq r$, subject to the relations (\ref{rel01}). Observe that $\left( \begin{smallmatrix} 1 & 0 \\ n & 1 \end{smallmatrix} \right) \in \Gamma(m,n)$. Hence this matrix is a product of the $M_i$ and one consequence of the relations (\ref{rel01}) is
$$\overline{\alpha_x} = \overline{\alpha_x} + n\overline{\alpha_y}.$$
We thus find that $n\overline{\alpha_y} = 0$. Similarly we find that $m\overline{\alpha_x} = 0$. Obviously we can rewrite the defining relations (\ref{rel01}) as
\begin{equation}\label{rel02}
(a_i - 1)\overline{\alpha_x} = c_i\overline{\alpha_y}, \quad (1-d_i)\overline{\alpha_y} = b_i\overline{\alpha_x}.
\end{equation}
By definition of $\Gamma(m,n)$ we have
$$(a_i-1) \equiv b_i \equiv 0 \mod m, \quad (1-d_i) \equiv c_i \equiv 0 \mod n$$
for $1 \leq i \leq r$ so that all relations in (\ref{rel02}) are consequences of $m\overline{\alpha_x}=0$ and $n\overline{\alpha_y} = 0$. Hence we obtain a presentation
$$\Gamma^+(G,\pi)^{\ab} = \langle \overline{\alpha_x}, \overline{\alpha_y}, \overline{\varphi_1},\dots, \overline{\varphi_r} \mid \mbox{abelian},\ m\overline{\alpha_x}= 0,\ n\overline{\alpha_y}= 0 \rangle.$$
This proves $\Gamma^+(G,\pi)^{\ab} \cong G \times \Z^r$. The three special cases can be obtained by an explicit computation.

\subsection{Proof of Theorem \ref{main2}}
Let $G$ be a finite non-perfect group. First we consider the case that $G/G' \cong \Z / 2\Z$ where $G'$ denotes the commutator subgroup of $G$. Then we naturally obtain an epimorphism
$$\bar\pi: F_2 \stackrel{\pi}{\longrightarrow} G \longrightarrow \Z / 2\Z.$$
One easily verifies that
\begin{equation}\label{sub}
\Gamma^+(G,\pi) \leq \Gamma^+(\Z / 2\Z, \bar\pi).
\end{equation}
By \cite[Lem.~3.1]{AR} there is some $\varphi \in \Aut^+(F_2)$ such that $\bar\pi\varphi(x) = 1$ and $\bar\pi\varphi(y) = 0$. Observe that $\Gamma^+(G,\pi\varphi) = \varphi^{-1}\Gamma^+(G,\pi)\varphi$. We may therefore assume that $\bar\pi(x) = 1$ and $\bar\pi(y)=0$.  We set
$$\bar \rho : \Aut^+(F_2) \stackrel{\rho}{\longrightarrow} \SL_2(\Z) \longrightarrow \PSL_2(\Z)$$
where the second epimorphism is the natural projection. Note that $\bar\rho$ is onto. Since $\bar\rho$ induces an epimorphism $\Gamma^+(G,\pi)^{\ab} \rightarrow \bar\rho(\Gamma^+(G,\pi))^{\ab}$, it suffices to show that $\bar\rho(\Gamma^+(G,\pi))$ has infinite abelianization. By (\ref{sub}) we have 
$$\bar\rho(\Gamma^+(G,\pi)) \leq \bar\rho(\Gamma^+(\Z / 2\Z,\bar\pi)) = \PG^1(2).$$ 
Hence $\bar\rho(\Gamma^+(G,\pi))$ is a finite index subgroup of $\PG^1(2)$. Note that $\PG^1(2) = \PG^0(2)$. By an example of Rademacher \cite[Sec.~8]{RA}, we have
$$\PG^0(2) = \left\langle \begin{pmatrix} 1 & 0 \\ -1 & 1 \end{pmatrix} \right\rangle * \left\langle \begin{pmatrix} 1 & -2 \\ 1 & -1 \end{pmatrix} \right\rangle$$
where the the first factor is infinite cyclic and the second one has order~ $2$. The Kurosh Subgroup Theorem yields that $\bar\rho(\Gamma^+(G,\pi))$ is the free product of
\begin{itemize}
\item[(i)] a possibly trivial free group,
\item[(ii)] certain subgroups of conjugates of $\left\langle \left( \begin{smallmatrix} 1 & 0 \\ -1 & 1 \end{smallmatrix} \right) \right\rangle$,
\item[(iii)] certain conjugates of $\left\langle \left( \begin{smallmatrix} 1 & -2 \\ 1 & -1 \end{smallmatrix} \right) \right\rangle$.
\end{itemize}
We shall prove that a free factor of type (i) or (ii) actually appears. For a contradiction, let us assume that $\bar\rho(\Gamma^+(G,\pi))$ is the free product of factors of type (iii) only. Then $\bar\rho(\Gamma^+(G,\pi))$ is generated by elements of order $2$ and $\bar\rho(\Gamma^+(G,\pi))^{\ab} \cong (\Z / 2\Z)^m$ for some $m \in \N$. Let $k$ be the order of $G$. Then the automorphism
$\varphi \in \Aut^+(F_2)$ given by 
$$\varphi(x) = xy^k,\ \varphi(y) = y$$
is an element of $\Gamma^+(G,\pi)$. Hence 
$$M := \begin{pmatrix}1 & 0 \\ k & 1 \end{pmatrix} \in \bar\rho(\Gamma^+(G,\pi)).$$
 Since $M \in \left\langle \left( \begin{smallmatrix} 1 & 0 \\ -1 & 1 \end{smallmatrix} \right) \right\rangle$, one easily sees that the image of $M$ in $\PG^0(2)^{\ab}$ has infinite order. On the other hand, the image of $M$ in $\bar\rho(\Gamma(G,\pi))^{\ab}$ must have finite order. Observe that the inclusion map $\bar\rho(\Gamma^+(G,\pi)) \hookrightarrow \PG^0(2)$ induces a homomorphism $\bar\rho(\Gamma^+(G,\pi))^{\ab} \rightarrow \PG^0(2)^{\ab}$ such that the following diagram commutes.
$$ \xymatrix{\bar\rho(\Gamma^+(G,\pi)) \ar[r] \ar[d]&  \PG^0(2) \ar[d]\\
             \bar\rho(\Gamma^+(G,\pi))^{\ab} \ar[r] & \PG^0(2)^{\ab}}$$
This implies that the image of $M$ in $\PG^0(2)^{\ab}$ has finite order, contradiction.

The proof for $G^{\ab} = \Z / 3 \Z$ is almost the same. In the remaining cases, we find that $\bar\rho(\Gamma^+(G,\pi))$ is a finite index subgroup of a free subgroup of $\PSL_2(\Z)$. In particular, it has infinite abelianization and so must have $\Gamma^+(G,\pi)$.

\bibliographystyle{plain}

\end{document}